\documentclass{siamltex}
\usepackage[dvips]{graphics}
\usepackage{amssymb}
\usepackage{amsfonts,latexsym}
\usepackage{mathtools}
\usepackage{color}

\newcommand{\R}{{\mathbb R}}

\newcommand{\el}{\end{list}}
\newtheorem{algorithm}[theorem]{Algorithm}

\newtheorem{exmp}{Example}
\newtheorem{rmk}{Remark}
\newcommand{\balg}{\begin{algorithm}}
\newcommand{\ealg}{\end{algorithm}}
\newcommand{\bl}{\begin{list}{ \ }{
\leftmargin=.325in}}

\title{On the banded Toeplitz structured distance to symmetric positive 
semidefiniteness}\medskip
\author{
Silvia Noschese\thanks{Dipartimento di Matematica ``Guido Castelnuovo'', SAPIENZA 
Universit\`a di Roma, P.le A. Moro, 2, I-00185 Roma, Italy. E-mail: 
{\tt noschese@mat.uniroma1.it}. Research partially supported by a grant from SAPIENZA 
Universit\`a di Roma and by INdAM-GNCS.}
\and
Lothar Reichel\thanks{Department of Mathematical Sciences, Kent State University, Kent,
OH 44242, USA. E-mail: {\tt reichel@math.kent.edu}. Research partially supported by 
NSF grant DMS-1729509.}
}

\begin{document}
\maketitle

\begin{abstract}
This paper is concerned with the determination of a close real banded positive definite 
Toeplitz matrix in the Frobenius norm to a given square real banded matrix. While it is 
straightforward to determine the closest banded Toeplitz matrix to a given square matrix, the 
additional requirement of positive definiteness makes the problem difficult. We review 
available theoretical results and provide a simple approach to determine a banded
positive definite Toeplitz matrix.
\end{abstract}

\keywords
Matrix nearness problem, Toeplitz structure, symmetric positive definite matrix, 
structured distance, banded Toeplitz matrix 
\endkeywords

\section{Introduction}
Matrix nearness problems are the focus of much research in linear algebra; see, e.g., 
\cite{Dml87,Dml90,Hen,Hig,Ruh}. In particular, characterizations of the algebraic variety 
of normal matrices and distance measures to this variety have received considerable 
attention \cite{ElsIkr,ep,FKKL,Fr,GJSW,It,las,l1,l2,Ruh87}. Normal Toeplitz 
matrices are characterized in \cite{FKKL,GP,It}, and the distance of Toeplitz matrices to
the algebraic variety of normal Toeplitz matrices, measured with the Frobenius norm and 
referred to as the \emph{Toeplitz structured distance to normality}, is investigated in 
\cite{NR11}, where also an application to preconditioning is described. The present paper 
is concerned with the distance of banded Toeplitz matrices to the variety of similarly 
structured positive semidefinite matrices, and with determining the closest matrix in this
variety. 

We denote banded Toeplitz matrices in ${\R}^{n\times n}$ with bandwidth $2k+1$ by
\begin{equation}\label{BTM}
T_{(k)}=(n;k;\sigma,\delta,\tau)=
\left[
\begin{array}{cccccccc}
\delta & \tau_1  &\tau_2\ &\dots\ &\tau_{k}\ & \ & \ & \mbox{\huge 0} \\
\sigma_1 & \delta & \tau_1 &\ \ &\ &\ \ & \ &\ \\
\sigma_2& \sigma_1 & \ddots &\ & \ \ & \ &\ddots &\ \\
\vdots &\ &\ &\ &\ddots &\  \ & \ & \tau_{k}\ \\
\sigma_k&\ &\ddots &\ddots &\ddots \ & \ & \ & \vdots\  \\
&\ &\ &\ & \ &\ &\tau_1  & \tau_2 \\
& \ &\ddots &\ &\ &\sigma_1 &\delta & \tau_1 \\
\mbox{\huge 0} &\  & \ &\sigma_k\ &\dots\ &\sigma_2 & \sigma_1 &
\delta
\end{array}
\right].
\end{equation}
Some of the scalars $\sigma_j$, $\delta$, and $\tau_j$ may vanish. We say that the matrix 
(\ref{BTM}) has bandwidth $2k+1$, or equivalently, is $(2k+1)$-banded, even if $\sigma_k$ 
or $\tau_k$ vanish. All Toeplitz matrices in this paper are real. Banded Toeplitz matrices 
arise in many applications including signal processing, time-series analysis, and 
numerical methods for the solution of differential equations; see, e.g., 
\cite{BG,FGHLW,GS,Ka,TE}. Low-rank modifications of symmetric banded Toeplitz matrices are
considered in \cite{LP}.

Necessary and sufficient conditions for a banded Toeplitz matrix \eqref{BTM} to be normal 
are given in \cite{NR09}, where also the distance of $(2k+1)$-banded Toeplitz matrices of 
order $n$, with $k$ less than or equal to the integer part of $n/2$, to the algebraic 
variety of normal Toeplitz matrices of the same bandwidth, is investigated. The distance 
is measured with the Frobenius norm. Since the given matrix \eqref{BTM} and the closest 
normal Toeplitz matrix both are banded Toeplitz matrices, we refer to their distance as 
the \emph{banded Toeplitz structured distance to normality}. Whether this distance measure
is more meaningful than the Toeplitz structured distance depends on the application. 

Let $A=[a_{ij}]_{i,j=1}^n\in\R^{n\times n}$ and define the Frobenius norm 
\[
\|A\|_F=\sqrt{\sum_{i,j=1}^n a_{ij}^2}.
\]
The distance from $A$ to the set of symmetric positive semidefinite $n\times n$ 
matrices in the Frobenius norm is given by 
\begin{equation}\label{minE}
\delta_F^+(A):=\min\{ \|E\|_F:~E\in\R^{n\times n},~ A+E \text{ symmetric positive 
semidefinite} \}.
\end{equation}
This distance can be expressed as
\begin{equation}\label{deltaF}
\delta_F^+(A)= \sqrt{\sum_{\lambda_i(B)<0}\lambda_{i}(B)^{2}+\|C\|_F^2},
\end{equation}
where $B$ and $C$ are the symmetric and skew-symmetric parts of $A$, respectively, and 
$\lambda_1(B),\lambda_2(B),\dots, \lambda_n(B)$ are the eigenvalues of $B$. The nearest 
symmetric positive semidefinite matrix in the Frobenius norm is $A_+:=(B + H)/2$, with $H$
the symmetric polar factor of $B$ defined as follows: Consider the spectral factorization 
$B= Z\diag(\lambda_i(B))Z^T$  with $Z$ an orthogonal matrix. Then 
$H=Z\diag(|\lambda_i(B)|)Z^T$; see Higham \cite[Theorem 2.1]{Hig88}. Unfortunately, 
neither the matrices $A+E$ nor $A_+$ generally have the same structure as $A$. 
In fact, the polar factor of a symmetric banded Toeplitz 
matrix typically is neither Toeplitz nor banded.

We are interested in determining the {\it banded Toeplitz structured distance to symmetric 
positive semidefiniteness}, $\Delta_F^{+}$, as well as the projection, $T^{+}_{(k)}$, of a 
given banded Toeplitz matrix $T_{(k)}$ in the set of the similarly structured symmetric 
positive semidefinite matrices. The matrix $T^{+}_{(k)}$ has potential application to
preconditioning; see Section \ref{sec6}.  Our next example shows that the closest Toeplitz 
matrix in the Frobenius norm to a symmetric positive definite matrix might not be symmetric 
positive definite. Therefore, it is not straightforward to determine $\delta_F^+(A)$ and 
$\Delta_F^{+}$ even in this special situation.

\begin{exmp}\label{ex00}
Let 
\[
A=\left[\begin{array}{ccc} 100 & 99 & 0 \\ 99 & 100 & 1/2 \\ 0 & 1/2 & 1
\end{array}\right]. 
\]
The spectrum of $A$ is approximately $\{0.6461,~1.3532,~199.0006\}$. Thus, $A$ is 
symmetric positive definite. The closest Toeplitz  matrix in the Frobenius norm to $A$ is
obtained by averaging the  entries  of $A$ on every diagonal,
\[
T=\left[\begin{array}{ccc} 67 & 49.75 & 0 \\ 49.75 & 67 & 49.75 \\ 0 & 49.75 & 67
\end{array}\right],
\]
and has the spectrum $\{-3.3571,~67.0000,~137.357\}$. Hence, $T$ is symmetric indefinite.
\end{exmp}

This paper is organized as follows. Section \ref{sec2} discusses the structured distance 
of banded Toeplitz matrices to normality in the Frobenius norm and Section \ref{sec3} is 
concerned with the structured distance to the set of positive semidefinite Toeplitz 
matrices. Also an approach to inexpensively determine a banded symmetric positive definite
matrix that is close to a given banded Toeplitz matrix is described. The special case of 
tridiagonal matrices is considered in Section \ref{sec4}. 
Some remarks on the application of the symmetric positive definite banded Toeplitz matrices 
determined in Sections \ref{sec3} and \ref{sec4} to the solution of linear systems of 
equations are provided in Section \ref{sec6}. Concluding remarks can be found in Section 
\ref{sec8}.

\section{Structured distance to normality}\label{sec2}
This section reviews notation and results from \cite{NPR13,NR09} that will be useful in 
the sequel. 

\begin{theorem}(\cite{NR09})\label{T1}
The real $(2k+1)$-banded Toeplitz matrix $T_{(k)}=(n;k;\sigma,\delta,\tau)$ of order $n$
with $k\leq\lfloor n/2\rfloor$ is normal if and only if it is either symmetric or shifted 
skew-symmetric (i.e., obtained by adding to a skew-symmetric matrix a multiple of the 
identity). Consider the sum
\begin{equation}\label{Tcond}
\sum_{h=1}^{k}(n-h)\sigma_{h}\tau_{h}
\end{equation}
associated with the matrix $T_{(k)}=(n;k;\sigma,\delta,\tau)$. If this sum is positive, 
then the projection of $T_{(k)}$ onto the algebraic variety of similarly structured normal
matrices is the real symmetric $(2k+1)$-banded Toeplitz matrix
\[
T_{1,(k)}^{*}=(n;k;\frac{\sigma+\tau}{2},\delta,\frac{\sigma+\tau}{2}).
\]
If, instead, the sum (\ref{Tcond}) is negative, then the projection of $T_{(k)}$ onto the
algebraic variety of similarly structured normal matrices is the real shifted
skew-symmetric $(2k+1)$-banded Toeplitz matrix
\[
T_{2,(k)}^{*}=(n;k;\frac{\sigma-\tau}{2},\delta,\frac{\tau-\sigma}{2}).
\]
Finally, if the sum (\ref{Tcond}) vanishes, then both matrices $T_{1,(k)}^{*}$ and 
$T_{2,(k)}^{*}$ are closest matrices to $T_{(k)}$ in the algebraic variety of similarly 
structured normal matrices. Moreover, the squared banded Toeplitz structured distance to
normality from $T_{(k)}$ is 
\[
\Delta_F(T_{(k)})^2=\frac{1}{2}\min\left\{\sum_{j=1}^k(n-j)(\sigma_j-\tau_j)^2,
\sum_{j=1}^k(n-j)(\sigma_j+\tau_j)^2\right\}.
\]
\end{theorem}

Note that the matrices $T_{(k)}$ and 
$T_{(k)}-\delta I_n=(n;k;\sigma,0,\tau)\in \R^{n \times n}$ have the same 
banded Toeplitz structured distance to normality; however, they have different projections
onto  the algebraic variety of similarly structured normal matrices (at distance 
$\sqrt{n}|\delta|$).

Theorem \ref{T1} greatly simplifies in the tridiagonal case and the following result 
holds.

\begin{corollary}(\cite[Theorem 3.3]{NPR13})\label{T2}
The squared $3$-banded Toeplitz structured distance to normality from $T_{(1)}$ is 
\[
\Delta_F(T_{(1)})^2=
\frac{n-1}{2}\min\left\{(\sigma_1-\tau_1)^2,
(\sigma_1+\tau_1)^2\right\}=\frac{n-1}{2}\left| |\sigma_1|-|\tau_1| \right|^2.
\]
Moreover, the normal tridiagonal Toeplitz matrix closest to $T_{(1)}$ is the symmetric 
matrix  $T_{1,(1)}^{*}= (n;1;\frac{\sigma_1+\tau_1}{2},\delta,\frac{\sigma_1+\tau_1}{2})$
if $\sigma_1\tau_1\geq 0$, and the shifted skew-symmetric matrix 
$T_{2,(1)}^{*}= (n;1;\frac{\sigma_1-\tau_1}{2}, \delta,\frac{\tau_1-\sigma_1}{2})$ if 
$\sigma_1\tau_1\leq0$.
\end{corollary}

\section{Structured distance to symmetric positive semidefiniteness}\label{sec3}
It follows from \eqref{deltaF} that the squared distance of a real $(2k+1)$-banded 
(possibly nonsymmetric) Toeplitz matrix $T_{(k)}=(n;k;\sigma,\delta,\tau)$ to the set of
symmetric positive semidefinite matrices is 
\begin{equation} \label{dspd}
\delta_F^+(T_{(k)})^2=\sum_{\lambda_i^{(k)}<0}(\lambda_{i}^{(k)})^2+\sum_{i=1}^k 
\frac{n-i}{2}(\sigma_i-\tau_i)^2,
\end{equation}
where the $\lambda_i^{(k)}$ denote the eigenvalues of the closest symmetric matrix 
(in the Frobenius norm) to $T_{(k)}$. We note that the closest
symmetric matrix to $T_{(k)}$ in the Frobenius norm is 
$T_{1,(k)}^{*}=(n;k;\frac{\sigma+\tau}{2},\delta,\frac{\sigma+\tau}{2})$, but the 
latter matrix is not guaranteed to be positive definite. Moreover, 
$T_{2,(k)}^{*}-\delta I_n= (n;k;\frac{\sigma-\tau}{2},0,\frac{\tau-\sigma}{2})$ is the 
closest skew-symmetric matrix to $T_{(k)}$ (in the Frobenius norm) and
$$
\|T_{2,(k)}^{*}-\delta I_n\|_F^2=\sum_{i=1}^k \frac{n-i}{2}(\sigma_i-\tau_i)^2.
$$

We are interested in determining the closest symmetric positive semidefinite 
$(2k+1)$-banded Toeplitz matrix $T_{(k)}^+$ to $T_{(k)}$, as well as the distance 
\[
\Delta_F^+(T_{(k)}):=\|T_{(k)}-T_{(k)}^+\|_F.
\]
There is no simple expression available for $T_{(k)}^+$ even when $T_{(k)}$ is symmetric. 
We therefore seek to determine an approximation $\widetilde{T}_{(k)}^{+}$ of $T_{(k)}^+$, 
as well as the upper bound 
\[
\widetilde{\Delta}_F^{+}({T}_{(k)}):=\| T_{(k)}-\widetilde{T}_{(k)}^{+}\|_F
\]
for $\Delta_F^{+}({T}_{(k)})$. A natural approach to determine a suitable approximation 
$\widetilde{T}_{(k)}^{+}$ of $T_{(k)}^{+}$ is to shift the matrix $T_{1,(k)}^{*}$ so that
all eigenvalues of the shifted matrix are nonnegative and one of them is zero. The 
following well-known result (see, e.g., \cite{GS,BG}) provides an approach to choose the 
shift, which we denote by $\gamma$.

\begin{proposition}\label{Sz}
The set $\{\lambda_h^{(n)}\}_{h=1}^{n}$ of eigenvalues of the symmetric banded Toeplitz 
matrix $T_{(k)}=(n;k;\sigma,\delta,\sigma)$, ordered so that
$\lambda_1^{(n)}\geq\lambda_2^{(n)}\geq\ldots\geq\lambda_n^{(n)}$, 
are distributed as $\{g(\frac{h\pi}{n+2})\}_{h=1}^{n}$, where $g$ is the symbol for the
matrix $T_{(k)}$, i.e.,
\[
g(\theta)=\delta+2\sum_{j=1}^k \sigma_j \cos (j\theta), \quad \theta \in (-\pi,\pi).
\]
Moreover, if $g(\theta)\geq 0$, $\forall\theta\in(-\pi,\pi)$, then $T_{(k)}$ is positive 
semidefinite or positive definite.
\end{proposition}

\begin{rmk} 
Notice that the matrix $T_{(k)}=(n;k;\sigma,\delta,\sigma)$ is positive definite if its 
symbol $g(\theta)$ has only isolated zeros; see, e.g., \cite{S2}.
\end{rmk}

Let $T_{(k)}$ be a general banded Toeplitz matrix. One obtains a symmetric positive 
semidefinite $(2k+1)$-banded Toeplitz matrix by shifting the symmetric matrix 
$T_{1,(k)}^{*}$ by $\gamma I_n$, where 
\begin{equation}\label{gamma}
\gamma= \max \{ 0, \sum_{j=1}^k |\sigma_j + \tau_j |- \delta  \}.
\end{equation}
We remark that an application of Gershgorin disks suggests the same shift. 

The following result can be used to bound the distance between the spectra of 
$T_{1,(k)}^{*}$ and $\widetilde{T}^{+}_{(k)}:=T_{1,(k)}^{*}+\gamma I_n$, with $\gamma$ 
defined by \eqref{gamma}. 

\begin{proposition}(\cite{RB86})\label{dist_spec}
Let the matrices $A\in{\mathbb R}^{n\times n}$ and $B\in{\mathbb R}^{n\times n}$ be 
symmetric and measure the distance between the matrices $A$ and $B$ in the Frobenius 
norm, 
\[
d_{A,B}:=\|A-B\|_F.
\]
Let the eigenvalues of the matrices $A$ and $B$ be ordered to be nonincreasing as a
function of their index. Introduce the vectors 
\[
\lambda(A)=[\lambda_1(A),\lambda_2(A),\ldots,\lambda_n(A)]^T, \quad
\lambda(B)=[\lambda_1(B),\lambda_2(B),\ldots,\lambda_n(B)]^T,
\]
containing all eigenvalues $\lambda_j(A)$ of $A$ and $\lambda_j(B)$ of $B$, respectively,
and define the distance between these vectors by means of the Euclidean norm $\|\cdot\|$,
$$
d_{\lambda(A,B)}:=\|\lambda(A)-\lambda(B)\|.
$$
Then 
\begin{equation}\label{evbd}
d_{A,B} \geq d_{\lambda(A,B)}.
\end{equation}
\end{proposition}

In our context, we obtain equality in \eqref{evbd}, 
\[
d_{T_{1,(k)}^{*},\widetilde{T}^{+}_{(k)}}=\|T_{1,(k)}^{*}-\widetilde{T}^{+}_{(k)}\|_F=
\sqrt{n}\gamma=\|\lambda(T_{1,(k)}^{*})-\lambda(\widetilde{T}^{+}_{(k)})\|=
d_{\lambda(T_{1,(k)}^{*},\widetilde{T}^{+}_{(k)})}\,.
\]
We define for future reference 
\[
\widetilde{\Delta}_F^{+}(T_{1,(k)}^{*})=\|T_{1,(k)}^{*}-\widetilde{T}^{+}_{(k)}\|_F.
\]

\begin{rmk}\label{0n}
Notice that the matrix $\delta I_n$ may be considered a $(2k+1)$-banded Toeplitz matrix 
for any $k$ less than or equal to the integer part of $n/2$. It is symmetric positive 
semidefinite if and only if $\delta\geq 0$.
\end{rmk}

Using \eqref{deltaF}, the (unstructured) nearness to symmetric positive semidefiniteness 
of $T_{1,(k)}^{*}$ can be bounded by 
\begin{eqnarray*}
\delta_F^+(T_{1,(k)}^{*})^2&=&
\sum_{\lambda_i(T_{1,(k)}^{*})<0}\lambda_{i}(T_{1,(k)}^{*})^{2}\\
&\leq&\sum_{i=1}^n \lambda_{i}(T_{1,(k)}^{*})^{2}=
\|T_{1,(k)}^{*}\|_F^2=\sum_{i=1}^k\frac{n-i}{2}(\sigma_i+\tau_i)^2 + n\delta^2,
\end{eqnarray*}
where equality is attained when the spectrum of $T_{1,(k)}^{*}$ is confined to the 
negative real axis. In this case, the closest symmetric positive semidefinite matrix to 
$T_{1,(k)}^{*}$ is the zero matrix $O_n$, because 
\[
\max\{0,\delta\}=\max\left\{0,\sum_{i=1}^n\lambda_i(T_{1,(k)}^{*})\right\}=0.
\]
Thus, for $\Delta_F^+(T_{1,(k)}^{*})$, i.e., for the banded Toeplitz structured nearness 
to symmetric positive semidefiniteness of $T_{1,(k)}^{*}$, one obtains the lower and 
upper bounds
$$
\delta_F^+(T_{1,(k)}^{*})\leq \Delta_F^+(T_{1,(k)}^{*})\leq
\min \{\|T_{1,(k)}^{*}-\max\{0,\delta\}I_n\|_F,\widetilde{\Delta}_F^{+}(T_{1,(k)}^{*})\}.
$$

We conclude with the observation that, since $\|B+C\|_F^2=\|B\|_F^2+\|C\|_F^2$ if $B=B^T$ 
and $C=-C^T$, the above inequalities also hold for the banded Toeplitz structured distance
of $T_{(k)}$ to the set of symmetric positive semidefinite matrices. We have 
\begin{equation}\label{Delta_bounds}
\begin{array}{rcl}
\delta_F^+(T_{(k)})^2&\leq& \Delta_F^+(T_{(k)})^2\\
 &\leq& \min\{\|T_{1,(k)}^{*}-\max\{0,\delta\}I_n\|_F, 
\widetilde{\Delta}_F^{+}(T_{1,(k)}^{*})\}^2+\|T_{2,(k)}^{*}-\delta I_n\|^2,
\end{array}
\end{equation}
where equality is achieved if the spectrum of $T_{1,(k)}^{*}$ is confined to the negative
real axis, so that $\delta_F^+(T_{(k)})=\Delta_F^+(T_{(k)})=\|T_{(k)}\|_F$.

\begin{theorem}\label{upperb}
We have the following upper bounds for the squared $(2k+1)$-banded Toeplitz structured 
distance to symmetric positive semidefiniteness of $T_{(k)}=(n;k;\sigma,\delta,\tau)$:
$$
\min\left\{\sum_{i=1}^k (n-i)(\sigma_i^2+\tau_i^2),n\max\left\{0,\sum_{i=1}^k |\sigma_i+ 
\tau_i|-\delta\right\}^2+\sum_{i=1}^k \frac{n-i}{2}(\sigma_i-\tau_i)^2\right\}
$$
if $\delta>0$, and 
$$
\min\left\{\sum_{i=1}^k (n-i)(\sigma_i^2+\tau_i^2)+n\delta^2,n\left(\sum_{i=1}^k|\sigma_i+ 
\tau_i|-\delta\right)^2+\sum_{i=1}^k \frac{n-i}{2}(\sigma_i-\tau_i)^2\right\}
$$
if $\delta\leq0$.
These bounds can be computed in ${\cal O}(k)$ arithmetic floating point operations 
(flops).
\end{theorem}

\begin{proof}
The bounds follow from the upper bound in \eqref{Delta_bounds}, by replacing 
$\widetilde{\Delta}_F^{+}(T_{1,(k)}^{*})$ by $\sqrt{n}\gamma$, with $\gamma$ given by
\eqref{gamma}, and by observing that 
$$
\sum_{i=1}^k \frac{n-i}{2}(\sigma_i+\tau_i)^2+\sum_{i=1}^k\frac{n-i}{2}(\sigma_i-\tau_i)^2
= \sum_{i=1}^k (n-i)(\sigma_i^2+\tau_i^2).
$$ 
Moreover, it is 
straightforward to observe that the cost of computing these bounds increases linearly with
the bandwidth of the $(2k+1)$-banded Toeplitz matrix $T_{(k)}$. This concludes the proof.
\end{proof}

\begin{exmp}\label{ex1}
Consider the symmetric pentadiagonal Toeplitz matrices 
$T_{(2)}(p)=(100;2;\sigma,\delta,\tau)$, with entries $\sigma_1=\tau_1=0.05$, 
$\sigma_2=\tau_2=p$, and $\delta=0.1$, where $p$ ranges from $0.04$ to $0.06$ with step 
$0.0001$. Figure \ref{fig_ex} shows for each $p$ the squared distances 
$d_1=\sum_{i=1}^2\frac{n-i}{2}(\sigma_i+\tau_i)^2$ and 
$d_2=n\max \{0,\sum_{i=1}^2 |\sigma_i+\tau_i|-\delta\}^2$ in red and green, 
respectively. The squared distance $d_3=\sum_{i=1}^2 \frac{n-i}{2}(\sigma_i-\tau_i)^2$ 
always vanishes, since the matrices considered are symmetric.

It is easy to verify that the upper bounds in \eqref{Delta_bounds} for  
$\Delta_F^+(T_{(2)}(p))^2$ are attained by $d_2=\min\{d_1, d_2\}+d_3$ for $p$ ranging from
$0.04$ to $0.0492$ and by $d_1=\min\{d_1, d_2\}+d_3$ for $p$ ranging from $0.0493$ to 
$0.06$.
\end{exmp}

\begin{figure}
\centerline{
\includegraphics[scale=0.5,trim= 0mm 0.01mm 0mm 0mm]{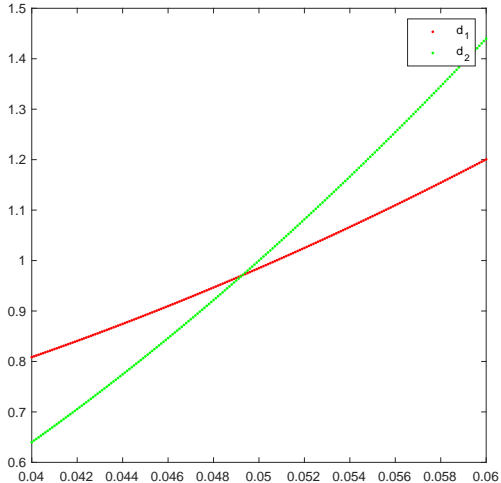}
}
\caption{Example \ref{ex1}. Upper bound for the $5$-banded Toeplitz structured nearness to 
symmetric positive semidefiniteness of $T_{(2)}(p)=(100;2;[0.05,p],0.1,[0.05,p])$, with 
$p=0.04:0.0001:0.06$.}\label{fig_ex}
\end{figure}

\section{Structured distance to symmetric positive semidefiniteness in the tridiagonal 
case}\label{sec4}
This section considers real tridiagonal Toeplitz matrices. The eigenvalues 
of such matrices are known to be
\[
\lambda_i(T_{(1)})=\delta + 2\sqrt{\sigma_1\tau_1}\,\cos \frac{i\pi }{n+1}, 
\;\;\;\;i=1,2,\dots,n.
\]
It is therefore possible to construct the symmetric tridiagonal 
semidefinite Toeplitz matrix $\widetilde{T}_{(1)}^{+}$ by adding the matrix $\gamma_n I_n$,
where $\gamma_n\geq0$ depends on the order $n$, to the symmetric part of $T_{(1)}$, i.e.,
to $T_{1,(1)}^*$, to obtain $\widetilde{T}_{(1)}^{+}$.

\subsection{The case $\mathbf{\boldsymbol{\delta}=0}$}
This subsection considers real tridiagonal Toeplitz matrices with vanishing diagonal 
entries. We may assume that the matrix $T_{(1)}=(n;1;\sigma_1,0,\tau_1)$ is of odd 
order $n$. The eigenvalues of $T_{1,(1)}^{*}$ then are given by 
\[
\lambda_i^{(1)}=|\sigma_1+\tau_1 |\,\cos \frac{i\pi }{n+1},  \;\;i=1,2,\dots,n.
\]
Thus, they are real and allocated symmetrically with respect to the origin, one of them 
vanishes; see, e.g., \cite{NPR13}. If, instead, $T_{(1)}$ is of even order, then no 
eigenvalue of $T_{1,(1)}^{*}$ vanishes.

The closest symmetric positive semidefinite (not necessarily tridiagonal Toeplitz) matrix 
computed by the algorithm in \cite{Hig88} has the same $(n-1)/2$ [or $n/2$, if $n$ is 
even] positive eigenvalues as 
$T_{1,(1)}^{*}$, whereas the remaining eigenvalues vanish. A matrix with such a spectrum 
cannot be a tridiagonal Toeplitz matrix; see, e.g., \cite{NPR13}. 

\begin{proposition}\label{dist_spec1}
The distance of the symmetric part $T_{1,(1)}^{*}$ of $T_{(1)}$ to symmetric positive 
semidefiniteness is 
\[
\delta_F^+(T_{1,(1)}^{*})= \frac{\sqrt{n-1}}{2}|\sigma_1+\tau_1 |.
\]
\end{proposition}

\begin{proof}
According to \eqref{dspd}, one has 
\[
\delta_F^+(T_{1,(1)}^{*})^2=\sum_{\lambda_i^{(1)}<0}(\lambda_{i}^{(1)})^2.
\]
The eigenvalues of $T_{1,(1)}^{*}$ are allocated symmetrically with respect to the 
origin. Therefore, 
\[
\delta_F^+(T_{1,(1)}^{*})^2 = \frac{1}{2} \|T_{1,(1)}^{*}\|_F^2 =
\frac{n-1}{4}(\sigma_1+\tau_1)^2.
\]
This concludes the proof.
\end{proof}

We are in a position to determine upper and lower bounds for $\Delta_F^+(T_{1,(1)}^{*})$.

\begin{corollary}\label{corx}
\begin{equation}\label{dist_spec2}
\frac{\sqrt{n-1}}{2}|\sigma_1+\tau_1 |\leq \Delta_F^+(T_{1,(1)}^{*})\leq 
\sqrt{\frac{{n-1}}{2}}|\sigma_1+\tau_1|
\end{equation}
for all $n=1,2,\ldots~$.
\end{corollary}

\begin{proof}
Theorem \ref{upperb} applied to $T_{1,(1)}^{*}=(n;1; \frac{\sigma_1+\tau_1}{2},0,
\frac{\sigma_1+\tau_1}{2})$ and Proposition \ref{dist_spec1} yield the lower and upper bounds
\[
\frac{\sqrt{n-1}}{2}|\sigma_1+\tau_1 |\leq \Delta_F^+(T_{1,(1)}^{*})\leq 
\min\left\{\sqrt{\frac{{n-1}}{2}} |\sigma_1+\tau_1 |, \sqrt{n}|\sigma_1+\tau_1 |\,
\right\},
\]
from which \eqref{dist_spec2} straightforwardly follows.
\end{proof}

Alternatively, one might consider shifting $T_{1,(1)}^{*}$ by a multiple of the 
identity so that all eigenvalues become nonnegative. Since the eigenvalues of 
$T_{1,(1)}^{*}$ are allocated symmetrically with respect to the origin, this means that we
could add a multiple of the identity, $\gamma_n I_n$, to $T_{1,(1)}^{*}$, with $\gamma_n$ 
equal to the spectral radius 
\[
\rho (T_{1,(1)}^{*})=|\sigma_1+\tau_1 |\,\cos \frac{\pi}{n+1}.
\] 
Notice that $\gamma_n=0$ if and only if $\sigma_1=-\tau_1$, i.e., $T_{1,(1)}^{*}=O_n$ and 
$T_{2,(1)}^{*}=T_{(1)}$. \\
This way, one would get
\[ 
\widetilde{\Delta}_F^{+}(T_{1,(1)}^{*})=\sqrt{n}|\sigma_1+\tau_1 |\,\cos \frac{\pi}{n+1}.
\]
However, it is easy to show that 
\begin{equation}\label{ineq}
\sqrt{\frac{{n-1}}{2}} |\sigma_1+\tau_1 |\leq \sqrt{n}|\sigma_1+\tau_1 |\,
\cos \frac{\pi }{n+1},
\end{equation}
for all $n=1,2,\ldots~$. Hence, the upper bound in \eqref{dist_spec2} is sharper. 
Indeed, direct computations show equality in \eqref{ineq} for $n\in\{1,2\}$ and, for 
$n\geq 3$, one has
\[
\frac{1}{2}\leq\cos^2\left(\frac{\pi}{n+1}\right).
\]

We next show lower and upper bounds for the $3$-banded Toeplitz structured distance of
$T_{(1)}=(n;1;\sigma_1,0,\tau_1)$ to symmetric positive semidefiniteness in the Frobenius 
norm. 

\begin{theorem} \label{upperb1}
For the squared $3$-banded Toeplitz structured distance to symmetric positive 
semidefiniteness of the matrix $T_{(1)}=(n;1;\sigma_1,0,\tau_1)$, we have the lower and
upper bounds 
\begin{equation}\label{bds}
\frac{n-1}{4}(3\,\sigma_1^2+3\,\tau_1^2-2\,\sigma_1\tau_1)\leq \Delta_F^+(T_{(1)})^2\leq
(n-1) (\sigma_1^2+\tau_1^2).
\end{equation}
\end{theorem}

\begin{proof}
The upper bound follows from 
\[
\Delta_F^+(T_{(1)})^2\leq\|T_{(1)}\|_F^2=(n-1)(\sigma_1^2+\tau_1^2).
\]
The lower bound is obtained from
\[
\Delta_F^+(T_{(1)})^2\geq\delta_F^+(T_{(1)})^2=\delta_F^+
(T_{1,(1)}^*)^2+\|T_{2,(1)}^*\|_F^2,
\]
where the equality is a consequence of \eqref{deltaF}. The first term on the right-hand 
side is given by Proposition \ref{dist_spec1} and the second term is evaluated in a 
straightforward manner to give the lower bound \eqref{bds}.
%
\end{proof}

\begin{exmp} \label{e1} 
Consider the downshift matrix
\begin{equation}\label{shiftmat}
T_{(1)}=(n;1;1,0,0)=\left[\begin{array}{cccccc}
0 & 0 & \cdots & \cdots & 0 & 0 \\
1 & 0 &0 & \cdots  & 0 & 0 \\
0 & 1 &0 & \cdots  & 0 & 0 \\
\vdots  & 0 & \ddots  & \vdots & \vdots & \vdots \\
 \vdots &  \vdots & \ddots  &   \ddots      & 0 & 0 \\
0 & \cdots  & \cdots  &  0  & 1 & 0 
\end{array}\right].\end{equation}
Its squared distances in the Frobenius norm to the sets of the symmetric and 
skew-symmetric (banded Toeplitz) matrices are both $\frac{n-1}{2}$. Thus, the squared 
banded Toeplitz structured distance to normality is $\frac{n-1}{2}$. It is shown in
\cite[Section 9]{NPR} that the squared (non-banded) Toeplitz structured distance to 
normality is $\frac{n-1}{n}$, and a circulant being at this distance is described.
Moreover, it is shown in \cite[Proposition 2.2]{GNNR} that the squared distance to
the set of the symmetric positive semidefinite matrices in the Frobenius norm is
$\delta_F^+(T_{(1)})^2=\frac{3(n-1)}{4}$. These results are consistent with Theorem 
\ref{upperb1}. 

Applying our approach to constructing an approximate nearest symmetric tridiagonal 
positive semidefinite matrix $\widetilde{T}^{+}_{(1)}$ to $T_{(1)}$, we obtain
\begin{equation}\label{posdefTo} 
\widetilde{T}^{+}_{(1)}=(n;1;\frac{1}{2},\cos \frac{\pi }{n+1},\frac{1}{2})=
\left[\begin{array}{ccccc}
\cos \frac{\pi }{n+1} & \frac{1}{2} & \cdots & \cdots  & 0 \\
\frac{1}{2} & \cos \frac{\pi }{n+1} &\frac{1}{2} & \cdots   & 0 \\
0 & \frac{1}{2} &\dots & \cdots   & 0 \\
\vdots  & 0 & \ddots  & \vdots & \vdots \\
 \vdots &  \vdots & \ddots  &   \ddots      &  \frac{1}{2}  \\
0 & \cdots  & \cdots  &  \frac{1}{2}  &  \cos \frac{\pi }{n+1} 
\end{array}\right].
\end{equation}
and 
\[
\widetilde{\Delta}_F^{+}(T_{(1)})^2=\|T_{(1)}-\widetilde{T}_{(1)}^{+}\|^2=
\|T_{1,(1)}^{*}-\widetilde{T}_{(1)}^{+}\|_F^2+\|T_{2,(1)}^*\|_F^2=
n\cos^2 \frac{\pi }{n+1}+\frac{n-1}{2}.
\]
Moreover, the squared distance between the spectrum of $T_{1,(1)}^{*}$ and the spectrum of
$\widetilde{T}_{(1)}^{+}$ is 
$\|\lambda^{(1)}-\widetilde{\lambda}^{+}\|^2=n\cos^2 \frac{\pi }{n+1}$, whereas the 
squared distance between the spectrum of $T_{1,(1)}^{*}$  and the spectrum of the 
(not necessarily tridiagonal Toeplitz) closest symmetric positive semidefinite matrix is 
$\|\lambda^{(1)}-\lambda^{+}\|^{2}=\frac{n-1}{4}$. Here $\lambda^{(1)}$ is the vector of
all eigenvalues of $T_{1,(1)}^{*}$ ordered nonincreasingly, $\widetilde{\lambda}^{+}$ 
denotes the vector of eigenvalues of $\widetilde{T}_{(1)}^{+}$ ordered in the same manner,
and $\lambda^{+}$ is a vector of all eigenvalues of the closest symmetric positive 
semidefinite matrix ordered similarly; $\|\cdot\|$ denotes the Euclidean norm.

Finally, regard the approximate nearest symmetric positive semidefinite tridiagonal 
Toeplitz matrix $O_n\in\R^{n\times n}$. The squared distance from $T_{(1)}$ to $O_n$ is 
$\| T_{(1)}\|_F^2=n-1$, whereas the squared distance  between the spectrum of 
$T_{1,(1)}^{*}$ and the spectrum of $O_n$ is  $\|\lambda^{(1)}\|^2 =\frac{n-1}{2}$. We 
obtain
\[
\frac{3(n-1)}{4}=\delta_F^+(T_{(1)})^2\leq\Delta_F^{+}(T_{(1)})^2\leq\| T_{(1)}\|_F^2=n-1.
\]
\end{exmp}
 
\subsection{The case $\mathbf{\boldsymbol{\delta}\neq0}$} 
We consider general tridiagonal Toeplitz matrices $T_{(1)}=(n;1;\sigma_1,\delta,\tau_1)$. 
The distance $\delta^+_F $ to symmetric positive semidefiniteness depends on the 
eigenvalues of $T_{1,(1)}^{*}$, which are given by 
\[
\lambda_i^{(1)}=\delta + |\sigma_1+\tau_1 |\,\cos \frac{i\pi }{n+1},\qquad i=1,2,\dots,n.
\]

When $\delta <|\sigma_1+\tau_1 |\,\cos \frac{\pi }{n+1}$, the closest symmetric positive 
semidefinite matrix computed as in \cite{Hig88} has the same nonnegative eigenvalues as  
${T}_{1,(1)}^{*}$ and the remaining eigenvalues are zero. If, instead,
$\delta \geq |\sigma_1+\tau_1 |\,\cos \frac{\pi }{n+1}$, then ${T}_{1,(1)}^{*}$ is
symmetric positive semidefinite and the following argument is not needed. 

Assume that the matrix ${T}_{1,(1)}^{*}$ is not symmetric positive semidefinite. Then we 
can add a multiple of the identity, $\gamma_n I_n$, to $T_{1,(1)}^{*}$ with 
\begin{equation}\label{gamma_n}
\gamma_n=|\delta -|\sigma_1+\tau_1 |\,\cos \frac{\pi }{n+1}|
\end{equation}\
so that the matrix $\widetilde{T}_{(1)}^{+}=T_{1,(1)}^{*}+\gamma_n I_n$ is positive 
semidefinite. We have 
\[
\widetilde{T}_{(1)}^{+}= (n;1;(\sigma_1+\tau_1)/2,|\sigma_1+\tau_1 |\,
\cos \frac{\pi }{n+1},(\sigma_1+\tau_1)/2).
\]
Alternatively, one may consider $\max\{0,\delta\}I_n$ as an approximate nearest symmetric 
positive semidefinite tridiagonal Toeplitz matrix.

\begin{theorem}\label{upperb1d} 
For the squared $3$-banded Toeplitz structured 
distance to symmetric positive semidefiniteness of the matrix $T_{(1)}=(n;1;\sigma,\delta,\tau)$,
we have the upper bounds
$$
\min\left\{(n-1)(\sigma_1^2+\tau_1^2), n\max\left\{0,|\sigma_1+\tau_1|\cos\frac{\pi}{n+1}- 
\delta\right\}^2+\frac{n-1}{2}(\sigma_1-\tau_1)^2\right\} 
$$
if $\delta>0$, and
$$
\min\left\{(n-1)(\sigma_1^2+\tau_1^2)+n\delta^2,n\left(|\sigma_1+\tau_1|
\cos\frac{\pi}{n+1}-\delta\right)^2+\frac{n-1}{2}(\sigma_1-\tau_1)^2\right\}
$$
if $\delta\leq0$. The cost of computing these upper bounds is ${\cal O}(1)$ flops.
\end{theorem}

\begin{proof}
The bounds follow from \eqref{Delta_bounds}, where 
$\widetilde{\Delta}_F^{+}(T_{1,(k)}^{*})$ is given by the shift $\gamma_n$ in 
\eqref{gamma_n} (instead of the shift $\gamma$ in \eqref{gamma}, as in Theorem 
\ref{upperb}). 
\end{proof}

\begin{exmp} \label{e3}
Consider the symmetric tridiagonal Toeplitz matrices 
$T_{(1)}(p)=(15;1;\sigma_1,\delta,\tau_1)$ with $\sigma_1=\tau_1=0.05$ and $\delta=p$, 
where $p$ ranges from $0.02$ to $0.05$ with step $0.0001$. The left-hand side graphs of
Figure \ref{fig_ex3} show, for each $p$, the squared distances 
\[
d_1=\frac{n-1}{2}(\sigma_1+\tau_1)^2,\quad d_2=n\max\{0,|\sigma_1+\tau_1|-\delta\}^2
\]
that come from  Theorem \ref{upperb}, in red and green, respectively. Since the 
squared distance $d_3=\frac{n-1}{2} (\sigma_1-\tau_1)^2$ vanishes, both $d_1$ and 
$d_2$ provide upper bounds for the squared $3$-banded Toeplitz structured distance 
to symmetric positive semidefiniteness. The graphs on the right-hand side of Figure 
\ref{fig_ex3} show, for each $p$, the squared distances 
\[
d_1=\frac{n-1}{2}(\sigma_1+\tau_1)^2,\quad 
d_2=n\max\{0,|\sigma_1+\tau_1|\cos\frac{\pi }{n+1}-\delta\}^2 
\]
that come from Theorem \ref{upperb1d}, in red and green, respectively. These distances
provide upper bounds for the squared $3$-banded Toeplitz structured distance to 
symmetric positive semidefiniteness. It is easy to verify that they are sharper.
For instance, for $p=0.03$, one has $d_1=0.0700$ and, in the former case, $d_2= 0.0735$ 
(left graph), whereas in the latter case $d_2= 0.0695$ (right graph). 
\end{exmp}

\begin{figure}
\centerline{
\includegraphics[scale=0.4,trim= 0mm 0.01mm 0mm 0mm]{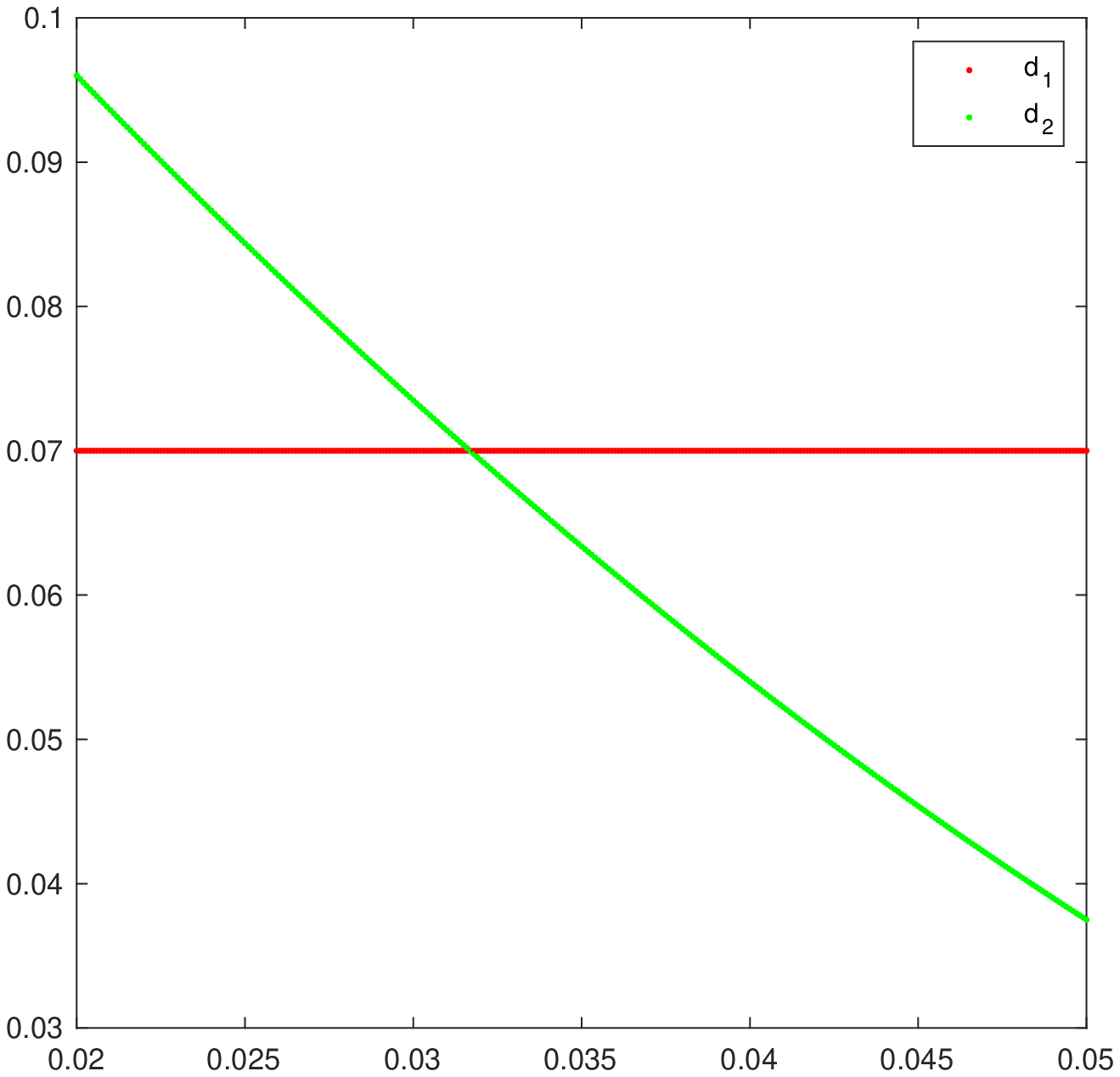}
\includegraphics[scale=0.4,trim= 0mm 0.01mm 0mm 0mm]{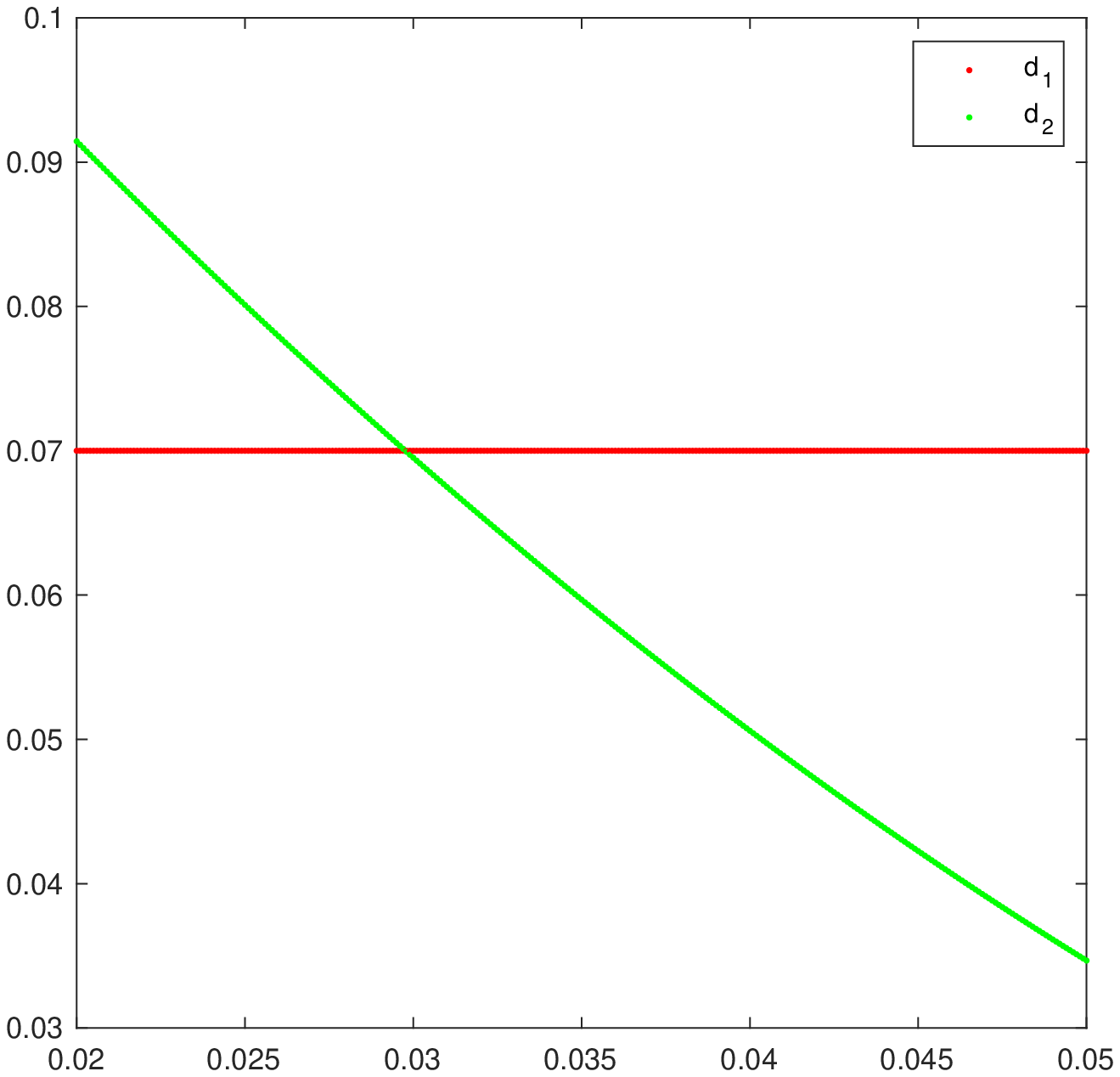}
}
\caption{Example \ref{e3}. Upper bounds for the tridiagonal Toeplitz structured distance 
to symmetric positive semidefiniteness of $T_{(1)}(p)=(15;1;0.05,p,0.05)$ for 
$p=0.02:0.0001:0.05$. In the left graph the upper bounds are determined by Theorem 
\ref{upperb}, whereas in the right graph they are determined by Theorem \ref{upperb1d}.}
\label{fig_ex3}
\end{figure}

\section{Remarks on applications to the solution of linear systems of equations}\label{sec6}
Consider the solution of a linear system of equations
\begin{equation}\label{linsys}
Ax=b,\qquad A\in\R^{n\times n},\quad x,b\in\R^n,
\end{equation}
with a large matrix $A$. This matrix is not required
to have any particular structure, but we assume that $A$ is close to a banded symmetric 
positive definite Toeplitz matrix $T^+$. It is natural to use the matrix $T^+$ as a
preconditioner, because linear systems of equations with a banded symmetric positive 
definite Toeplitz matrix can be solved rapidly and stably by exploiting the Toeplitz 
structure by Schur or generalized Schur algorithms described in 
\cite{AG0,AG,DG,Ka,LMV18}\footnote{The superfast generalized Schur algorithms described in
\cite{AG0,AG} do not exploit bandedness.}, as well as by the method by Bini and Meini 
\cite{BB}. 

When the matrix $A$ is symmetric positive definite, we can solve \eqref{linsys} by the
preconditioned conjugate gradient method using $T^+$ as a preconditioner; see, e.g., 
\cite[Algorithm 11.5.1]{GVL}. If $A$ is symmetric indefinite, then the conjugate gradient 
method should be replaced by the SYMMLQ algorithm; see \cite{PS} for a description of the 
latter. Finally, when $A$ is nonsymmetric, a preconditioned iterative method designed for
the solution of systems of equations with a nonsymmetric matrix should be used, such as
preconditioned GMRES; see \cite{GVL,Saad}. In all these situations, the preconditioned 
iterative methods are simpler when using a symmetric positive definite preconditioner, 
 because Schur and generalized Schur algorithms can be applied to rapidly solve 
linear systems of equations with such a preconditioner matrix.
Though, the application of an indefinite or nonsymmetric preconditioner for the solution of
Toeplitz systems also has been described in the literature; see \cite{CPS,HST,S1}.



Large banded matrices that are close to the set of banded symmetric positive definite 
Toeplitz matrices arise when discretizing second order differential equations in one space
dimension on the interval $0<t<1$ at equidistant grid points using the standard symmetric 
second order 3-point finite difference approximation of the second derivative $-d^2/dt^2$ 
with some boundary conditions. For instance, Neumann boundary conditions give rise to a 
symmetric tridiagonal matrix of the form 
\begin{equation}\label{singmat}
\frac{1}{h^2} (T_{(1)}- e_1 e_1^T-e_n e_n^T),
\end{equation}
where $e_j=[0,\ldots,0,1,0,\ldots,0]^T\in\R^{n}$ denotes the $j$th 
canonical  basis vector, $T_{(1)}=(n,1;-1,2,-1)$, and $h=1/(n+1)$. The matrix $T_{(1)}$ is 
symmetric positive definite, while the matrix \eqref{singmat} is singular. Consistent linear 
systems of equations with the latter matrix have a unique solution $x=[x_1,x_2,\ldots,x_n]^T$ 
such that $\sum_{j=1}^n x_j=0$. The matrix $T_{(1)}$ can be used as a preconditioner. 
Analogous formulas arise in higher space dimensions. 

Other techniques for determining positive definite banded Toeplitz preconditioners with
applications to the solution of partial differential equations are described by Chan 
\cite{Ch} and Hon et al. \cite{HSW}, who apply the Remez algorithm to compute a 
nonnegative low-degree trigonometric polynomial that approximates the symbol associated 
with a given symmetric indefinite Toeplitz matrix. This can be fairly expensive. The 
approach described in the present paper is much cheaper, but may give a banded Toeplitz 
matrix of lower quality. A careful comparison of these approaches to construct 
preconditioners is a topic of future work.

\section{Conclusion}\label{sec8}
This paper discusses the determination of a positive definite banded Toeplitz matrix
that is close to a Toeplitz matrix with the same band structure. A simple fast method
is described.

\section*{Acknowledgment}
The authors would like to thank Greg Ammar and a referee for comments that lead to
clarifications of the presentation.

\end{document}